\documentclass[a4paper]{article}

\usepackage[utf8]{inputenc}
\usepackage[english]{babel}
\usepackage[T1]{fontenc}
\usepackage{amsfonts}
\usepackage{amsmath}
\usepackage{amsthm}
\usepackage{amssymb}
\usepackage{graphicx}
\usepackage{graphics}
\usepackage{a4wide}

\usepackage{mathptmx} 
\usepackage{bm}

\usepackage{nicefrac}    
\usepackage{microtype}   
\usepackage{booktabs}    

\let\originalleft\left
\let\originalright\right
\renewcommand{\left}{\mathopen{}\mathclose\bgroup\originalleft}
\renewcommand{\right}{\aftergroup\egroup\originalright}

\usepackage{pifont}
\newcommand{\cmark}{\ding{51}}%
\newcommand{\xmark}{\ding{55}}

\usepackage{comment}
\usepackage{todonotes}
\usepackage[normalem]{ulem}
\usepackage{etoolbox}
\usepackage{color}
\newtoggle{FinalVersion}
\toggletrue{FinalVersion}
\iftoggle{FinalVersion}{
 
 \newcommand{\deleted}[1]{{}}
}{
 
 \newcommand{\deleted}[1]{{\sout{\color{red}#1}}} 
}

\usepackage[section]{algorithm}
\usepackage{algorithmicx}
\usepackage{algpseudocode}

\algdef{SE}[DOWHILE]{Do}{doWhile}{\algorithmicdo}[1]{\algorithmicwhile\ #1}%

\newtheorem{thm}{Theorem}[section]
\newtheorem{theorem}[thm]{Theorem}
\newtheorem{lemma}[thm]{Lemma}
\newtheorem{proposition}[thm]{Proposition}

\theoremstyle{definition}

\newtheorem{assumption}[thm]{Assumption}

\newcommand{\norm}[1]{\|#1\|}
\newcommand{\nrm}[1]{|#1|}

\newcommand{\R}{{\mathbb R}}
\newcommand{\EE}{{\mathbb E}}

\newcommand{\bpm}{\begin{pmatrix}}
\newcommand{\epm}{\end{pmatrix}}

\newcommand{\mnmz}{\operatorname*{minimize}}
\newcommand{\mxmz}{\operatorname*{maximize}}
\newcommand{\st}{\operatorname{subject\ to}}

\newcommand{\eps}{{\varepsilon}}

\newcommand{\clip}{\operatorname{clip}}

\usepackage{tikz}
\usepackage{pgfplots}
\usepackage{pgfplotstable}

\newcommand{\Pcal}{{\mathcal P}}
\newcommand{\niter}{{n_{h}}}

\newcommand{\cmin}{{c_{\rm min}}}
\newcommand{\cmax}{{c_{\rm max}}}
\newcommand{\csmax}{{c_{\rm max2}}}

\usetikzlibrary{pgfplots.groupplots}
\pgfplotsset{compat=1.3}
\pgfplotsset{
    line1/.style={%
        blue},
    line2/.style={%
        blue, dashdotted},
    line3/.style={%
        black},
    line4/.style={%
        black, dashdotted},
    line5/.style={%
        red, solid},
    row1/.style={%
        xmin=0, xmax=17,
        ymin=70, ymax=100,
        xtick distance=4},
    row2/.style={%
        xmin=0, xmax=17,
        ymin=70, ymax=100,
        xtick distance=4},
    row3/.style={%
        xmin=0, xmax=19,
        ymin=70, ymax=100,
        xtick distance=4},
    legendStyleA/.style={%
        column sep = 10pt,
        legend to name = grouplegendA},
    legendStyleB/.style={%
        column sep = 10pt,
        legend to name = grouplegendB},
    legendStyleC/.style={%
        column sep = 10pt,
        legend to name = grouplegendC},
}

\pgfplotstableread[col sep = comma]{Results_Time.csv}\tabTime
\pgfplotstableread[col sep = comma]{Results_Iterations.csv}\tabIter
\pgfplotstableread[col sep = comma]{Results_h.csv}\tabH

\usepackage{authblk}

\title{Projections onto the canonical simplex with additional linear inequalities}
\date{November 6, 2019}

\author[1,2]{Luk\'a\v{s} Adam\thanks{adam@utia.cas.cz}}
\author[2]{V. M\'acha}

\affil[1]{Southern University of Science and Technology, Shenzhen 518055, China}
\affil[2]{\'UTIA, The Czech Academy of Sciences, Pod Vod\'arenskou v{\v ez}\'i 4, 18208, Prague, Czech Republic}

\begin{document}

\maketitle

\begin{abstract}
We consider the distributionally robust optimization and show that computing the distributional worst-case is equivalent to computing the projection onto the canonical simplex with additional linear inequality. We consider several distance functions to measure the distance of distributions. We write the projections as optimization problems and show that they are equivalent to finding a zero of real-valued functions. We prove that these functions possess nice properties such as monotonicity or convexity. We design optimization methods with guaranteed convergence and derive their theoretical complexity. We demonstrate that our methods have (almost) linear observed complexity.

\end{abstract}

\section{Introduction}

The projection of a vector onto the unit simplex appears in various fields such as portfolio optimization \cite{markowitz1952portfolio}, multi-phase physics \cite{bendsoe2013topology}, mathematical optimization \cite{nelder1965simplex}, knapsack problem \cite{kellerer2013knapsack} or machine learning applications \cite{blondel2014large}. Given a vector $\bm q\in\R^n$, this projection amounts to solving
\begin{equation}\label{eq:problem_simplex}
\aligned
\mnmz_{\bm p}\ &\frac12\norm{\bm p-\bm q}^2 \\
\st\ &\sum_{i=1}^np_i=1,\\
&0\le p_i,\quad\forall i=1,\dots,n.
\endaligned
\end{equation}
The Karush-Kuhn-Tucker optimality conditions imply that if one solves
\begin{equation}\label{eq:system0}
\sum_{i=1}^n\max\{q_i-\mu,0\} = 1
\end{equation}
for $\mu$, then one recovers the optimal solution of \eqref{eq:problem_simplex} by thresholding
\begin{equation}\label{eq:threshold0}
p_i=\max\{q_i-\mu,0\}.
\end{equation}
This was discovered for the first time in  \cite{held1974validation} and then rediscovered many times later \cite{maculan1989linear}. The simplest way to solve \eqref{eq:system0} is to sort $\bm q$, derive an iterative procedure computing the whole function on the left-hand side of \eqref{eq:system0} and then find when its value equals to $1$. Since the second part can be done in $O(n)$, the whole algorithm has complexity $O(n\log n)$ due to the sorting.

This procedure was improved in numerous papers. \cite{van2008probing} observed that only those $q_i$ above $\mu$ need to be sorted in \eqref{eq:system0}. Using a partially sorted structure called heap, they managed to reduce the complexity to $O(n+k\log n)$, where $k$ is the number of $q_i$ above the optimal $\mu$. \cite{kiwiel2008breakpoint} realized that many operations in quicksort may be ignored when it is used to solve \eqref{eq:system0} and reached complexity $O(n)$. \cite{michelot1986finite} proposed a simple method based on the fixed-point theorem with observed complexity $O(n)$. \cite{condat2016fast} provided an excellent overview, pointed to some errors in previous papers and designed an improved algorithm.

Besides the standard projection \eqref{eq:problem_simplex}, multiple versions appear in the literature. \cite{adam2018semismooth} considered an infinite-dimensional optimization problem with partial differential equations in the constraints. To get the independence of the number of iterations on the mesh size, they derived a path-following algorithm. \cite{liu2009efficient} considered sparse learning problems containing a modified simplex with two vectors of variables whose sum had to be equal. They derived an improved bisection method and a fast-converging subgradient algorithm. A similar simplex appeared in \cite{shalev2006efficient} for ranking labels based on feedback and in \cite{Li_TopPush} for a special binary classification problem. \cite{lapin.2015} considered SVM with top-${\rm k}$ error instead of the standard top-$1$ error. The resulting modified simplex contained a variable upper bound which was not fixed as in all previous cases. They penalized one constraint and computed an approximate projection. Note that these problems are difficult as observed in \cite{adam2016normally}. \cite{philpott2018distributionally} considered a maximization of a linear function on reduced simplex. Their application came from financial stochastic dual dynamic programming.

In this paper, we also consider projections onto a modified simplex. Our motivation stems from the field of distributionally robust optimization \cite{delage2010distributionally} where one hedges against uncertainty by estimating a distribution $\bm q$ and considering the worse outcome when the true distribution $\bm p$ is not far away from $\bm q$. Since $\bm p$ is a distribution and we need to keep close to $\bm q$, this may be equivalently written as a projection with additional linear inequality. The projection is taken with respect to the distance between distributions and the additional linear inequality comes from considering the worst case. There are several possibilities of the considered distance function. Probably the most commonly used are the $\phi$-divergence \cite{bayraksan2015data} and the Wasserstein distance \cite{gao2016distributionally}. The authors in \cite{kapteyn2018distributionally,philpott2018distributionally} proposed an algorithm with quadratic complexity for the $l_2$ norm distance. In \cite{rahimian2019identifying} the authors provided a closed-form algorithm for $l_1$ distance.

The resulting problem is convex and depending on the used distance function, it may be even quadratic or linear. This suggests using general-purpose solvers such as CPLEX which is guaranteed to converge. In our paper, we perform a comparison with CPLEX and IPOPT and show that our algorithm exhibit the observed (almost) linear complexity and outperform the above general-purpose solvers.

The paper is organized as follows. In Section \ref{sec:problem} we give a brief introduction into the field of distributionally robust optimization. We derive the problems of interest and focus on measuring the distance by various $\phi$-divergences and $l_p$ norms. In Section \ref{sec:results} we present the main results. Instead of penalizing one constraint as in \cite{lapin.2015}, we write the full KKT system and simplify it into one equation in one variable similar to \eqref{eq:system0}. We derive the thresholding operator similar to \eqref{eq:threshold0}. For readability, we postpone all proofs to the Appendix. In Section \ref{sec:bounds} we consider numerical properties and show that the derived equations have nice properties such as monotonicity or convexity. Finally, in Section \ref{sec:numerics} we focus on numerical results. All codes are available online.\footnote{\texttt{https://github.com/VaclavMacha/Projections}\label{footnote:codes}}

\section{Motivation from distributionally robust optimization}\label{sec:problem}

In this section, we provide motivation for the problems from the field of distributionally robust optimization. In the classical robust optimization, one minimizes a random function $f(\bm x,\bm\xi)$ with respect to a decision variable $\bm x$ while considering the worse possible random outcome of a random variable $\bm \xi$ which is bound to lie in $\Xi$. This leads to the problem
\begin{equation}\label{eq:robust1}
\mnmz_{\bm x}\quad \mxmz_{\bm \xi\in\Xi} f(\bm x,\bm \xi).
\end{equation}
Since \eqref{eq:robust1} considers the worst possible scenario, this approach is usually too conservative. One way to alleviate it, is to consider the distributionally robust optimization, where one takes the worst outcome with respect to all probability distributions and not to all scenarios. Denoting the probability distribution by $P$, expectation with respect to $P$ by $\EE_P$ and the set of all admissible probability distributions by $\Pcal$, this results in
\begin{equation}\label{eq:robust2}
\mnmz_{\bm x}\quad \mxmz_{P\in\Pcal} \EE_Pf(\bm x,\bm \xi).
\end{equation}
Note that if $\Pcal$ consists of all Dirac measures concentrated at $\Xi$, then \eqref{eq:robust1} and \eqref{eq:robust2} coincide.

The simplest case appears when we know possible realization $\bm \xi_i$ for the random variable and each may happen with the probability $p_i$. Then the inner maximization problem in \eqref{eq:robust2} reduces to
\begin{equation}\label{eq:robust3}
\mxmz_{\bm p\in\Pcal} \EE_Pf(\bm x,\bm \xi) = \mxmz_{\bm p\in\Pcal} \sum_{i=1}^np_if(\bm x,\bm \xi_i) = \mxmz_{\bm p\in\Pcal} \bm c^\top \bm p,
\end{equation}
where we set $c_i:= f(\bm x,\bm \xi_i)$. However, the probability distribution $\bm p$ is often not known. Then we may assume that $\bm p$ is close to some known estimate $\bm q$ and we may want to hedge against the worst possible small deviation from $\bm q$. Measuring this deviation bu $\hat d$, the inner problem in \eqref{eq:robust2} then takes form
\begin{equation}\label{eq:problem0}
\aligned
\mxmz_{\bm p}\ &\bm c^\top \bm p \\
\st\ &\sum_{i=1}^np_i=1,\\
&0\le p_i,\quad\forall i=1,\dots,n,\\
&\hat d(\bm p, \bm q)\le\eps,
\endaligned
\end{equation}
The first two constraints prescribe that $\bm p$ is a probability distribution while the last one determines that $\bm p$ is not far from $\bm q$.

\subsection{Connection to projection onto the canonical simplex}\label{sec:convex}

Since convex programming allows to switch the objective and constraints, for each $\eps$ there is some $\delta$ such that problem \eqref{eq:problem0} is equivalent to
\begin{equation}\label{eq:problem0_equiv}
\aligned
\mnmz_{\bm p}\ &\hat d(\bm p, \bm q)& \\
\st\ &\sum_{i=1}^np_i=1,\\
&0\le p_i,\quad\forall i=1,\dots,n,\\
&\bm c^\top \bm p\ge\delta.
\endaligned
\end{equation}
This implies that problem \eqref{eq:problem0} is equivalent to a projection (with respect to distance $\hat d$) onto the canonical simplex restricted by additional linear constraint. In Section \ref{sec:simplex} we will consider a similar problem where the canonical simplex is restricted by upper bounds.

\subsection{Notation and assumptions}

In the manuscript, we employ the following notation and assumption:
$$
\aligned
\cmax &:= \max_{i=1,\dots,n} c_i, \\
\cmin &:= \min_{i=1,\dots,n} c_i, \\
I &:= \{i\mid c_i=\cmax\}.
\endaligned
$$
\begin{assumption}\label{ass}
We consider the following assumptions:
\begin{enumerate}\itemsep 0pt
\item[(A1)] Vector $\bm q$ has positive components which sum to one.
\item[(A2)] Vector $\bm c$ is not a constant vector.
\end{enumerate}
\end{assumption}

\noindent Assumption (A1) is natural since we want to measure the distance of $\bm p$ from $\bm q$ which is a distribution. Assumption (A2) is technical only. If $\bm c$ is a constant vector, then the objective $\bm c^\top \bm p=\cmax$ is constant and the optimization is trivial.

\section{Reduction of projections onto modified simplex to  one equation}\label{sec:results}

In the previous section, we mentioned how a projection onto the unit simplex with an additional constraint arises in the field of distributionally robust optimization. In the introduction, we recalled a way of solving the projection onto the unit simplex \eqref{eq:system0}. Namely, one needs to solve the equation \eqref{eq:system0} and then apply the thresholding operator \eqref{eq:threshold0} to obtain the solution.

In this section, we follow a similar approach to solve problem \eqref{eq:problem0} with two types of the distance function $\hat d$. Moreover, we use the same technique to derive an algorithm for projection onto the canonical simplex with additional upper bounds. These results allow us to propose numerical methods with linear complexity.


\subsection{Distributionally robust optimization with $\phi$-divergences}

In this section we consider the distributionally robust optimization where the distance function $\hat d$ is a $\phi$-divergence. In this case
\begin{equation}\label{eq:phi_div}
\hat d(\bm p,\bm q) = \sum_{i=1}^n d(p_i,q_i) = \sum_{i=1}^n q_i\phi\left(\frac{p_i}{q_i}\right),
\end{equation}
where $d$ is a $\phi$-divergence and $\phi$ is the convex generating function with $\phi(1)=0$. Then the distributionally robust problem \eqref{eq:problem0} takes form
\begin{equation}\label{eq:problem1}\tag{DRO1}
\aligned
\mxmz_{\bm p}\ &\bm c^\top \bm p \\
\st\ &\sum_{i=1}^np_i=1,\\
&0\le p_i,\quad\forall i=1,\dots,n,\\
&\sum_{i=1}^n q_i\phi\left(\frac{p_i}{q_i}\right) \le\eps,
\endaligned
\end{equation}
Some examples of $\phi$-divergences are listed in Table \ref{table:divergences}. We also consider the variation distance ($l_1$ norm). Even though it is a $\phi$-divergence, we handle it in Section \ref{sec:norms} as it is a norm as well. We start with the following result which states that if $\eps$ is large enough, then the solution to \eqref{eq:problem1} is trivial.

\begin{table}[!ht]
\centering
\caption{Examples of the $\phi$-divergences $d$ and generating functions $\phi$. Note that $\phi$ is convex with $\phi(1)=0$.}
\label{table:divergences}
\begin{tabular}{@{}lll@{}}\toprule
Name & Generating function $\phi$ & Formula $d$ \\\midrule
Kullback-Leibler divergence & $\phi_1(t) = t\log t$ & $d_1(p,q) = p\log\left(\frac p q\right)$ \\
Burg entropy & $\phi_2(t) = -\log t$ & $d_2(p,q) = q\log\left(\frac q p\right)$ \\
Hellinger distance & $\phi_3(t) = (\sqrt t - 1)^2$ & $d_3(p,q) = \left(\sqrt p - \sqrt q\right)^2$ \\
$\chi^2$-distance & $\phi_4(t) = \frac1t (t - 1)^2$ & $d_4(p,q) = \frac{(p-q)^2}{p}$ \\
Modified $\chi^2$-distance & $\phi_5(t) = (t - 1)^2$ & $d_5(p,q) = \frac{(p-q)^2}{q}$ \\
\bottomrule
\end{tabular}
\end{table}

\begin{theorem}\label{thm0}
Let Assumption \ref{ass} hold true and define vector $\hat{\bm p}$ with components
\begin{equation}\label{eq:check1}
\hat p_i = \begin{cases} \frac{1}{\sum_{j\in I}q_j}q_i &\text{if }i\in I, \\ 0 &\text{otherwise.} \end{cases}
\end{equation}
If this solution satisfies
\begin{equation}\label{eq:check2}
\sum_{i=1}^n q_i\phi\left(\frac{\hat p_i}{q_i}\right) \le\eps
\end{equation}
then $\hat{\bm p}$ is the optimal solution of \eqref{eq:problem1}.
\end{theorem}

Theorem \ref{thm0} provides the optimal solution for the case of large $\eps$. In the opposite case, we provide the solution in the following list of theorems. Each of them handles one $\phi$-divergence from Table \ref{table:divergences}.

\begin{theorem}[Kullback-Leibler divergence]\label{thm1}
Assume that Assumption \ref{ass} holds true, that \eqref{eq:check2} is violated and that $\eps <- \log\left(\sum_{i\in I}q_i\right)$. Then there exists some $\mu\in (0,\frac{\cmax-\cmin}{\eps}]$ which solves
\begin{equation}\label{eq:thm1_1}
h_1(\mu) := \sum_{i=1}^nq_i \exp\left(\frac{c_i}{\mu}\right)\left(\frac{c_i}{\mu} - \log\left(\sum_{j=1}^nq_j\exp\left(\frac{c_j}{\mu}\right)\right) - \eps\right) = 0.
\end{equation}
Moreover, defining the non-normalized weights
\begin{equation}\label{eq:thm1_2}
\hat p_i = q_i \exp\left(\frac{c_i}{\mu}\right),
\end{equation}
the optimal solution of \eqref{eq:problem1} with the Kullback-Leibler divergence $\phi=\phi_1$ equals to $p_i=\frac{\hat p_i}{\sum \hat p_j}$. 
\end{theorem}

\begin{theorem}[Burg entropy]\label{thm2}
Assume that Assumption \ref{ass} holds true and that \eqref{eq:check2} is violated. Then there exists some $\lambda \in (\cmax, \cmax+\frac{\cmax-\cmin}{\eps}]$ which solves
\begin{equation}\label{eq:thm2_1}
h_2(\lambda) := \sum_{i=1}^n q_i\log\left(\lambda-c_i\right) + \log\left(\sum_{i=1}^n\frac{q_i}{\lambda-c_i}\right) - \eps = 0.
\end{equation}
Moreover, defining the non-normalized weights
\begin{equation}\label{eq:thm2_2}
\hat p_i = q_i\frac{1}{\lambda-c_i},
\end{equation}
the optimal solution of \eqref{eq:problem1} with the Burg entropy distance $\phi=\phi_2$ equals to $p_i=\frac{\hat p_i}{\sum \hat p_j}$.
\end{theorem}

\begin{theorem}[Hellinger distance]\label{thm3}
Assume that Assumption \ref{ass} holds true, that \eqref{eq:check2} is violated and that $\eps<2-2\sqrt{\sum_{i\in I}q_i}$. Then there exists some $\lambda \in (\cmax, \cmax+\frac{(2-\eps)(\cmax-\cmin)}{\eps}]$ which solves
\begin{equation}\label{eq:thm3_1}
h_3(\lambda) := 2\sum_{i=1}^n \frac{q_i}{\lambda-c_i} - (2-\eps)\sqrt{\sum_{i=1}^n \frac{q_i}{(\lambda-c_i)^2}} = 0.
\end{equation}
Moreover, defining the non-normalized weights
\begin{equation}\label{eq:thm3_2}
\hat p_i = q_i \frac{1}{(\lambda-c_i)^2},
\end{equation}
the optimal solution of \eqref{eq:problem1} with the Hellinger distance $\phi=\phi_3$ equals to $p_i=\frac{\hat p_i}{\sum \hat p_j}$.
\end{theorem}

\begin{theorem}[$\chi^2$-distance]\label{thm4}
Assume that Assumption \ref{ass} holds true and that \eqref{eq:check2} is violated. Then there exists some $\lambda > \cmax$ which solves
\begin{equation}\label{eq:thm4_1}
h_4(\lambda) := \left(\sum_{i=1}^nq_i\sqrt{\lambda-c_i}\right) \left(\sum_{i=1}^nq_i\frac{1}{\sqrt{\lambda-c_i}}\right) - 1 - \eps = 0.
\end{equation}
Moreover, defining the non-normalized weights
\begin{equation}\label{eq:thm4_2}
\hat p_i = q_i\frac{1}{\sqrt{\lambda-c_i}},
\end{equation}
the optimal solution of \eqref{eq:problem1} with the $\chi^2$-distance $\phi=\phi_4$ equals to $p_i=\frac{\hat p_i}{\sum \hat p_j}$.
\end{theorem}

\begin{theorem}[Modified $\chi^2$-distance]\label{thm5}
Assume that Assumption \ref{ass} holds true and that \eqref{eq:check2} is violated. Then there exists some $\lambda > -\cmax$ which solves
\begin{equation}\label{eq:thm5_1}
h_5(\lambda) := \sum_{i=1}^nq_i {\max}^2\left(c_i+\lambda,0\right) - (1+\eps)\left(\sum_{i=1}^nq_i \max\left(c_i+\lambda,0\right)\right)^2 = 0.
\end{equation}
Moreover, defining the non-normalized weights
\begin{equation}\label{eq:thm5_2}
\hat p_i = q_i \max\left(c_i+\lambda,0\right),
\end{equation}
the optimal solution of \eqref{eq:problem1} with the modified $\chi^2$-distance $\phi=\phi_5$ equals to $p_i=\frac{\hat p_i}{\sum \hat p_j}$.
\end{theorem}

\subsection{Distributionally robust optimization with norms}\label{sec:norms}

Anothor possibility to measure the distance between $\bm p$ and $\bm q$ is to use $l_p$ norms. This results in the following problem
\begin{equation}\label{eq:problem2}\tag{DRO2}
\aligned
\mxmz_{\bm p}\ &\bm c^\top \bm p \\
\st\ &\sum_{i=1}^np_i=1,\\
&0\le p_i,\quad\forall i=1,\dots,n,\\
&\norm{\bm p - \bm q}_p\le\eps,
\endaligned
\end{equation}
Note that the $l_1$ norm also generates a $\phi$-divergence but we handle it here. The simple algorithms for solving \eqref{eq:problem2} with the $l_1$ and $l_\infty$ norms are presented in Appendix \ref{app:norms}. Here, we show the results for the $l_2$ norm.

\begin{theorem}[$l_2$ norm]\label{thm6}
Let Assumption \ref{ass} hold true and define vector $\hat{\bm p}$ with components
\begin{equation}\label{eq:check3}
\hat p_i = \begin{cases} q_i + \frac{1}{\nrm{I}} - \frac{1}{\nrm{I}}\sum_{i\in I}q_j &\text{if }i\in I, \\ 0 &\text{otherwise.} \end{cases}
\end{equation}
If this solution satisfies
\begin{equation}\label{eq:check4}
\norm{\hat{\bm p} - \bm q}\le\eps,
\end{equation}
then $\hat{\bm p}$ is the optimal solution of \eqref{eq:problem2}. If $\hat {\bm p}$ violates \eqref{eq:check4}, then there exists some $\mu>0$ and $\lambda$ which solve
\begin{subequations}\label{eq:thm6_1}
\begin{align}
\label{eq:thm6_11} \sum_{i=1}^n{\min}^2\left(\lambda-c_i, \mu q_i\right) - \eps^2\mu^2&= 0, \\
\label{eq:thm6_12} \sum_{i=1}^n\min\left(\lambda-c_i, \mu q_i\right) &= 0.
\end{align}
\end{subequations}
Moreover, the optimal solution of \eqref{eq:problem2} with the $l_2$ norm equals to
\begin{equation}\label{eq:thm6_2}
p_i = \max\left(q_i - \frac{1}{\mu}(\lambda-c_i), 0\right).
\end{equation}
\end{theorem}

Unlike in the previous cases, system \eqref{eq:thm6_1} consists of two equations. To reduce them to one, we define first a function $g_6(\lambda;\mu)$ of $\lambda$ with fixed parameter $\mu$ by
\begin{equation}\label{eq:defin_g1}
g_6 (\lambda;\mu) := \sum_{i=1}^n\min\left(\lambda-c_i, \mu q_i\right).
\end{equation}
Lemma \ref{lemma:g6} states that for each $\mu>0$, there is unique $\lambda$ solving $g_6(\lambda;\mu)=0$. We stress this dependence of $\lambda$ on $\mu$ by writing $\lambda(\mu)$. Moreover, the same lemma states that $\lambda(\mu)$ is a continuous function. Defining the continuous function
\begin{equation}\label{eq:defin_f1}
h_6(\mu) := \sum_{i=1}^n{\min}^2\left(\lambda(\mu)-c_i, \mu q_i\right) - \eps^2\mu^2,
\end{equation}
we observe that solving system \eqref{eq:thm6_1} can be reduced to solving the single equation $h_6(\mu)=0$.

\subsection{Projection onto simplex with additional linear equality}\label{sec:simplex}

In Section \ref{sec:convex} we argued that \eqref{eq:problem2} with $p=2$ is equivalent to projecting onto the canonical simplex with additional linear inequality. In this section, we consider one more problem of the projection onto the simplex with additional upper bounds. This problem in a slightly more general form reads
\begin{equation}\label{eq:problem3}\tag{SIMPLEX}
\aligned
\mnmz_{\bm p}\ &\frac12\norm{\bm p-\bm q}^2 \\
\st\ &\sum_{i=1}^n p_i = 1, \\
&l_i\le p_i\le u_i
\endaligned
\end{equation}
We obtain the reduced optimality conditions as follows, where $\clip$ is the projection operator.

\begin{theorem}[Simplex with upper bounds]\label{thm7}
Assume that the feasible set of \eqref{eq:problem3} is non-empty. Then there exists some $\lambda\in\R$ which solves
\begin{equation}\label{eq:thm7_1}
h_7(\lambda) := \sum_{i=1}^n \clip_{[l_i,u_i]}(q_i-\lambda) - 1 = 0.
\end{equation}
Moreover, the optimal solution of \eqref{eq:problem3} equals to
\begin{equation}\label{eq:thm7_2}
p_i = \clip_{[l_i,u_i]}(q_i-\lambda).
\end{equation}
\end{theorem}

\noindent Function $h_7$ is piecewise linear and non-increasing in $\lambda$. This allows us to find a simple algorithm to find the solution, we present it in Algorithm \ref{alg:simplex}.

\section{Numerical considerations}\label{sec:bounds}

In the previous section, we derived theoretical results which will be the bases for numerical methods for solving \eqref{eq:problem1}, \eqref{eq:problem2} and \eqref{eq:problem3}. In this section, we introduce these numerical methods and derive their complexity.

\subsection{Computation of $\lambda$}\label{sec:lambda}

For problems \eqref{eq:problem1} including the $\phi$-divergences and for \eqref{eq:problem3}, we reduced the optimality conditions into one equation in one variable. For \eqref{eq:problem2} with $l_2$ norm, we reduced it into two equations in two variables from which $\lambda$ is implicitly computed, further reducing the system into one equation in one variable. It is not difficult to show that this implicit equation \eqref{eq:thm6_12} is equivalent to
\begin{equation}\label{eq:reduction1}
\sum_{i=1}^n \max\left(\mu q_i+c_i-\lambda,0\right) - \mu = 0.
\end{equation}
Since \eqref{eq:reduction1} is identical to \eqref{eq:system0}, there are algorithms in $O(n)$ which compute $\lambda$ for any fixed $\mu$. For simplicity, we implemented a simpler algorithm which first sorts $\mu \bm q+\bm c$ and then finds $\lambda$ in one pass through the sorted array.

For the analysis of \eqref{eq:problem3} we realize that $h_7$ is a piecewise linear function which is non-increasing in $\lambda$. Since this problem differs from \eqref{eq:system0} only by the upper bound, we conjecture that there is an algorithm solving it in $O(n)$. Here, we present an algorithm with complexity $O(n\log n)$. We have
$$
h_7\big(\min_i(q_i - u_i)\big) = \sum_{i=1}^n u_i - 1\ge 0.
$$
The inequality holds due to the assumption that the feasible set of \eqref{eq:problem3} is non-empty.

Denote $\bm s$ the sorted version of $\bm q-\bm l$ and $\bm r$ the sorted version of $\bm q-\bm u$. The main idea of Algorithm \ref{alg:simplex} is to utilize the piecewise linearity of $h_7$ with kinks at $s_i$ and $r_j$ by tracking the current slope $\hat a$. Algorithm \ref{alg:simplex} is an iterative procedure where at every iteration, we know the values of $h_7(s_{i-1})$ and $h_7(r_{j-1})$ and we want to evaluate $h_7$ at the next point. If $r_j\le s_i$, then we consider $\lambda=r_j$ and increase $j$ by one. Since a new point enters the active set which contributes to the slope, we increase the slope $\hat a$ by $1$. If $r_j > s_i$, then we consider $\lambda=s_i$ and increase $i$ by one. Since one point leaves the active set, we decrease the slope $\hat a$ by $1$. In both cases, $g$ is decreased by $\hat a$ times the difference between the old value and the new values of $\lambda$. Once $g$ decreases below $0$, we stop the algorithm and linearly interpolate between the last two values. To prevent an overflow, we set $r_{n+1}=\infty$. Concerning the initial values, since $r_1 < s_1$, we set $i=1$ and $j=2$.

\begin{algorithm}[H]
    \centering
    \caption{For computing $\lambda$ from \eqref{eq:thm7_1}}
    \label{alg:simplex}
    \begin{algorithmic}[1]
\State Sort $\bm q-\bm l$ into $\bm s$ and $\bm q-\bm u$ into $\bm r$\State $i\gets1$, $j\gets2$, $\hat a\gets 1$
\State $\lambda\gets r_1$, $g\gets \sum_{i=1}^n u_i-1$
\While{$g>0$}
\If {$r_j\le s_i$}
\State $g\gets g - \hat a(r_j - \lambda)$
\State $\hat a\gets \hat a + 1$
\State $\lambda\gets r_j$, $j\gets j+1$
\Else
\State $g\gets g - \hat a(s_i - \lambda)$
\State $\hat a\gets \hat a - 1$
\State $\lambda\gets s_i$, $i\gets i+1$
\EndIf
\EndWhile
\State \textbf{return} linear interpolation of the last two values of $\lambda$
    \end{algorithmic}
\end{algorithm}

\subsection{Numerical methods}

From the proofs of Theorems \ref{thm1}-\ref{thm3} we obtain
\begin{equation}\label{eq:bounds}
\aligned
\lim_{\mu\downarrow 0}\ h_1(\mu) &= +\infty, &&& h_1\left(\frac{\cmax-\cmin}{\eps}\right) &\le 0,\\
\lim_{\lambda\downarrow \cmax} h_2(\lambda) &= +\infty,  &&& h_2\left(\cmax+\frac{\cmax-\cmin}{\eps}\right) &\le 0, \\
\lim_{\lambda\downarrow \cmax} h_3(\lambda) &= -\infty,  &&& h_3\left(\cmax+\frac{(2-\eps)(\cmax-\cmin)}{\eps}\right) &\ge 0.
\endaligned
\end{equation}
This implies that the bisection method is convergent for solving $h_1(\mu)=0$, $h_2(\lambda)=0$ and $h_3(\lambda)=0$ when starting from the bounds suggested by \eqref{eq:bounds}.

To solve $h_4(\lambda)=0$, $h_5(\lambda)=0$ and $h_6(\mu)=0$ we first observe that convexity is present due to the following result.

\begin{proposition}\label{prop:convex}
We have the following:
\begin{itemize}\itemsep 0pt
\item If the assumptions of Theorem \ref{thm4} are satisfied, then $h_4$ is decreasing and convex on $(\cmax,\infty)$.
\item If the assumptions of Theorem \ref{thm5} are satisfied, then $h_5$ is positive on $(-\cmax,\lambda_0)$ and decreasing and concave on $(\lambda_0,\infty)$ for some $\lambda_0$.
\item If the assumptions of Theorem \ref{thm6} are satisfied, then $h_6$ is positive on $(0,\mu_0)$ and decreasing and concave on $(\mu_0,\infty)$ for some $\mu_0$.
\end{itemize}
\end{proposition}

\begin{table}[!ht]
\centering
\caption{Table showing which problems have an exact algorithm and for which problems, the bisection and Newton's methods are convergent. Note that all problems have at least one convergent algorithm.}
\label{table:convergence}
\begin{tabular}{@{}llll@{}}\toprule
& & \multicolumn{2}{c@{}}{Guaranteed convergence} \\\cmidrule{3-4}
& Exact algorithm & Bisection & Newton \\\midrule
\eqref{eq:problem1} with Kullback-Leibler divergence & \xmark & \cmark & \xmark\\
\eqref{eq:problem1} with Burg entropy & \xmark & \cmark & \xmark\\
\eqref{eq:problem1} with Hellinger distance & \xmark & \cmark & \xmark \\
\eqref{eq:problem1} with $\chi^2$-distance & \xmark & \cmark & \cmark \\
\eqref{eq:problem1} with Modified $\chi^2$-distance & \xmark & \cmark & \cmark \\
\eqref{eq:problem2} with $l_1$ norm & \cmark & $\cdot$ & $\cdot$\\
\eqref{eq:problem2} with $l_2$ norm & \xmark & \cmark & \cmark \\
\eqref{eq:problem2} with $l_\infty$ norm & \cmark & $\cdot$ & $\cdot$\\
\eqref{eq:problem3} & \cmark & $\cdot$ & $\cdot$\\
\bottomrule
\end{tabular}
\end{table}

Lemma \ref{lemma:newton} states that if we start with a point with $h_4(\lambda)>0$, $h_5(\lambda)<0$ or $h_6(\mu)<0$, respectively, the Newton's method is convergent. We summarize this discussion in Tables \ref{table:convergence} and \ref{table:properties}. The former shows which problem has an exact algorithm and which needs to be solved via an iterative method. Note that convergence is guaranteed for each problem. The latter comments more on the iterative methods and summarizes the equations needed to solve \eqref{eq:problem1} and \eqref{eq:problem2} with the $l_2$ norm. Moreover, it provides the bounds within which the solution lies and shows whether the function in question possesses convexity. 

\begin{table}[!ht]
\centering
\caption{Properties for the bisection and Newton's method. The table shows which equation needs to be solved for \eqref{eq:problem0}, the bounds and whether the function $h$ is convex.}
\label{table:properties}
\begin{tabular}{@{}llll@{}}\toprule
& Equation & Bounds & Convex \\\midrule
\eqref{eq:problem1} with Kullback-Leibler divergence & $h_1(\mu)=0$ & $0<\mu\le \frac{\cmax-\cmin}{\eps}$ & \xmark\\
\eqref{eq:problem1} with Burg entropy & $h_2(\lambda)=0$ & $\cmax < \lambda \le \cmax+\frac{\cmax-\cmin}{\eps}$ & \xmark\\
\eqref{eq:problem1} with Hellinger distance & $h_3(\lambda)=0$ & $\cmax < \lambda \le \cmax+\frac{(2-\eps)(\cmax-\cmin)}{\eps}$ & \xmark \\
\eqref{eq:problem1} with $\chi^2$-distance & $h_4(\lambda)=0$ & $\cmax < \lambda$ & \cmark\\
\eqref{eq:problem1} with Modified $\chi^2$-distance & $h_5(\lambda)=0$ &  $-\cmax < \lambda $ & \cmark \\
\eqref{eq:problem2} with $l_2$ norm & $h_6(\mu)=0$ & $0<\mu$ & \cmark \\
\bottomrule
\end{tabular}
\end{table}

We summarize the whole procedure in Algorithm \ref{alg:schema}. The bisection method may be initialized based on the bounds from Table \ref{table:properties} while the Newton's method must be initialized based on the paragraph following Proposition \ref{prop:convex}. 

\begin{algorithm}[H]
    \centering
    \caption{For solving \eqref{eq:problem1} with any $\phi$-divergence and \eqref{eq:problem2} with $l_2$ norm}
    \label{alg:schema}
    \begin{algorithmic}[1]
\State Compute $\hat{\bm p}$ from \eqref{eq:check1} or \eqref{eq:check3}, respectively
\If{$\hat{\bm p}$ satisfies \eqref{eq:check2} or \eqref{eq:check4}, respectively}
\State The optimal distribution $\bm p$ equal to $\hat{\bm p}$ 
\Else
\If{Kullback-Leibler divergence \textbf{or} Burg entropy \textbf{or} Hellinger distance}
\State Solve $h=0$ using the bisection method
\ElsIf{$\chi^2$-distance \textbf{or} Modified $\chi^2$-distance \textbf{or} $l_2$ norm}
\State Solve $h=0$ using the Newton's method
\EndIf
\State
Compute the optimal distribution $\bm p$ from the corresponding Theorem \ref{thm1}-\ref{thm6}
\EndIf
    \end{algorithmic}
\end{algorithm}

\subsection{Complexity}

In Table \ref{table:complexity1} we show the complexity of the evaluation of $h_1,\dots,h_6$ and the total complexity of the algorithm. For \eqref{eq:problem1}, the evaluation of $h_1,\dots,h_5$ has the complexity of $O(n)$. Similarly for \eqref{eq:problem2} with the $l_2$ norm, for every $\mu$, the computation of $\lambda(\mu)$ can be done in $O(n)$ as shown is Section \ref{sec:lambda}. Thus, the evaluation of $h_6$ consumes $O(n)$ as well. We get the total complexity by multiplying this by the number of evaluations $\niter$ of $h$. In order to have a good performance, the number of evaluations $\niter$ needs to stay constant. This happened in our numerical experiments as the second row of Figure \ref{fig:results} shows. Moreover, $\niter$ is guaranteed to be constant for the bisection method whenever the bracketing interval stays constant. In the table we also included problem \eqref{eq:problem2} with $l_1$ and $l_\infty$ norms and \ref{eq:problem3}. We provide a comparison of the theoretical and the observed complexity in Table \ref{table:complexity2} later.

\begin{table}[!ht]
\centering
\caption{Computational complexity evaluating function $h$ and the total complexity for solving problems \eqref{eq:problem1}, \eqref{eq:problem2} and \eqref{eq:problem3}. Here, $\niter$ refers to the number of evaluation of the function $h$ which is guaranteed to be constant for the bisection method whenever its bracketing interval from Table \ref{table:properties} is uniformly bounded. The observed complexity is shown in Table \ref{table:complexity2} later.}
\label{table:complexity1}
\begin{tabular}{@{}lll@{}}\toprule
& Evaluation of $h$ & Total \\\midrule
\eqref{eq:problem1} & $O(n)$ & $O(\niter n)$ \\
\eqref{eq:problem2} with $l_1$ norm & $-$ & $O(n\log n)$ \\
\eqref{eq:problem2} with $l_2$ norm & $O(n)$ & $O(\niter n)$ \\
\eqref{eq:problem2} with $l_\infty$ norm & $-$ & $O(n\log n)$ \\
\eqref{eq:problem3} & $-$ & $O(n\log n)$ \\
\bottomrule
\end{tabular}
\end{table}

\section{Numerical results}\label{sec:numerics}

In this section, we present the numerical results. We recall that our codes are available online.$^{\ref{footnote:codes}}$ In Section \ref{sec:results}, we derived the monotonicity and convexity of functions $h_1,\dots,h_7$ corresponding to problems \eqref{eq:problem1}, \eqref{eq:problem2} and \eqref{eq:problem3} and in Section \ref{sec:bounds}, we argued that finding a zero of these functions should be easy. This is confirmed in Figure \ref{fig:h}. We see that $h_1$, which corresponds to \eqref{eq:problem1} with Kullback-Leibler divergence, is decreasing and seems to be convex. Similarly, $h_6$ corresponding to \eqref{eq:problem2} with $l_2$ norm is first increasing and after approximately $\mu=14$ decreasing and concave. The convexity for $h_1$ was not proven while the concavity for $h_6$ follows from Proposition \ref{prop:convex}.

\begin{figure}[!ht]
\begin{tikzpicture}
 \pgfplotsset{small,width=8cm,samples=30}
 \begin{groupplot}[group style = {group size = 2 by 1, horizontal sep = 22pt}, grid=major, grid style={dotted, gray!50}]
 \nextgroupplot[title={\eqref{eq:problem1} with Kullback-Leibler divergence}, xlabel={$\mu$}, ylabel={Function $h(\mu)$}]
    \addplot [line1] table[x index=0, y index=1] {\tabH};
    \addplot [line4, dashed] table[x index=0, y index=2] {\tabH};    
 \nextgroupplot[title={\eqref{eq:problem2} with $l_2$ norm}, xlabel={$\mu$}, ylabel={}]
    \addplot [line1] table[x index=3, y index=4] {\tabH};
    \addplot [line4, dashed] table[x index=3, y index=5] {\tabH};    
 \end{groupplot}
\end{tikzpicture}
\caption{Functions $h_1(\mu)$ and $h_6(\mu)$. Finding zeros of these points is equivalent to solving \eqref{eq:problem1} with Kullback-Leibler divergence and for \eqref{eq:problem2} with $l_2$ norm.} 
\label{fig:h}
\end{figure}
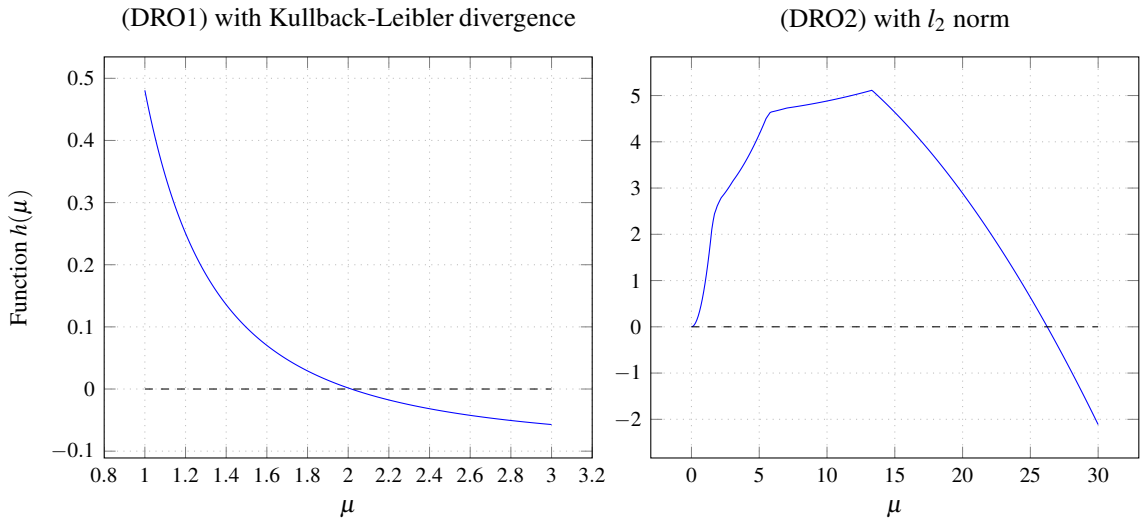

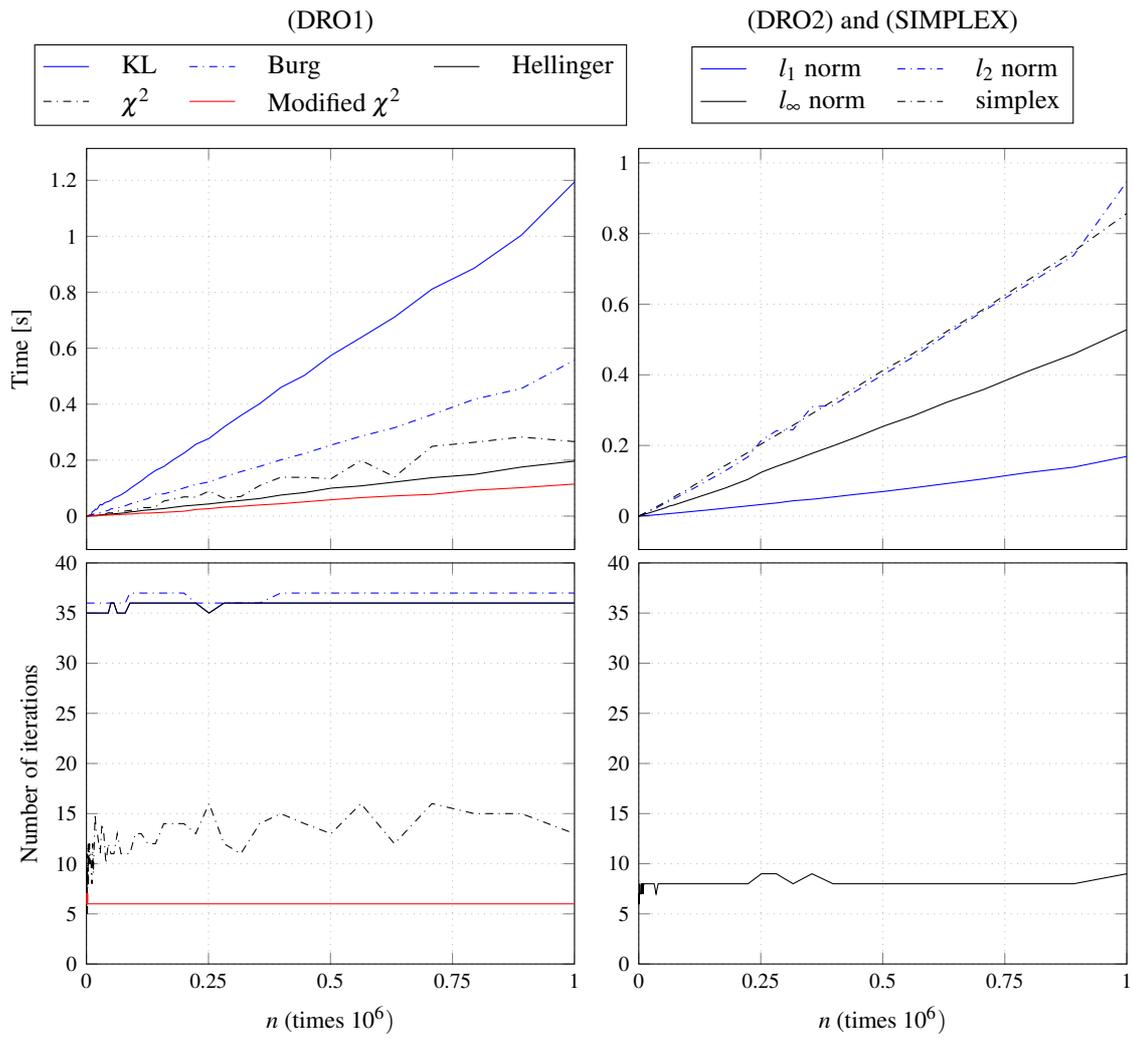
\begin{figure}
\begin{tikzpicture}
 \pgfplotsset{small,width=8cm,samples=30}
 \begin{groupplot}[xmin=0, xmax=1000000, scaled x ticks = false, xtick={0,250000,500000,750000,1000000}, xticklabels={0,0.25,0.5,0.75,1}, group style = {group size = 2 by 2, vertical sep = 5pt, horizontal sep = 24pt}, grid=major, grid style={dotted, gray!50}, legend cell align={left}]
 \nextgroupplot[align=center, title={\eqref{eq:problem1}\\ \\ \\}, ylabel={Time [s]}, xticklabels={}, legend style = legendStyleA, legend columns=3]
    \addplot [line1] table[x index=0, y index=2] {\tabTime};\addlegendentry{KL};
    \addplot [line2] table[x index=0, y index=6] {\tabTime};
    \addlegendentry{Burg};
    \addplot [line3] table[x index=0, y index=4] {\tabTime};
    \addlegendentry{Hellinger};
    \addplot [line4] table[x index=0, y index=8] {\tabTime};
    \addlegendentry{$\chi^2$};
    \addplot [line5] table[x index=0, y index=10] {\tabTime};
    \addlegendentry{Modified $\chi^2$};
 \nextgroupplot[align=center, title={\eqref{eq:problem2} and \eqref{eq:problem3} \\  \\ \\}, ylabel={}, xticklabels={}, legend style = legendStyleB, legend columns=2]
    \addplot [line1] table[x index=0, y index=14] {\tabTime};
    \addlegendentry{$l_1$ norm};
    \addplot [line2] table[x index=0, y index=16] {\tabTime};
    \addlegendentry{$l_2$ norm};
    \addplot [line3] table[x index=0, y index=12] {\tabTime};
    \addlegendentry{$l_\infty$ norm};
    \addplot [line4] table[x index=0, y index=19] {\tabTime};
    \addlegendentry{simplex};
 \nextgroupplot[align=center, title={}, ylabel={Number of iterations}, xlabel={$n$ (times $10^6)$}, ymin=0, ymax=40]
    \addplot [line1] table[x index=0, y index=2] {\tabIter};
    \addplot [line2] table[x index=0, y index=3] {\tabIter};
    \addplot [line3] table[x index=0, y index=4] {\tabIter};
    \addplot [line4] table[x index=0, y index=5] {\tabIter};
    \addplot [line5] table[x index=0, y index=6] {\tabIter};
 \nextgroupplot[align=center, title={}, ylabel={}, xlabel={$n$ (times $10^6)$}, ymin=0, ymax=40]
    \addplot [line3] table[x index=0, y index=7] {\tabIter};
 \end{groupplot}
 \node at ($(group c1r1) + (0,3.5)$) {\ref{grouplegendA}};  
 \node at ($(group c2r1) + (0,3.5)$) {\ref{grouplegendB}};  
\end{tikzpicture}
\caption{Performance of our methods for \eqref{eq:problem1} (left) and for \eqref{eq:problem2} and \eqref{eq:problem3} (right) for $n\in[10^3,10^6]$. The first row shows the measured times in seconds while the second row show the number of evaluations of $h(\mu)$ or $h(\lambda)$.}
\label{fig:results}
\end{figure}

For numerical comparison, we randomly generated the initial data $\bm q$ and $\bm c$ and solved problems \eqref{eq:problem1}, \eqref{eq:problem2} and \eqref{eq:problem3}. This was repeated hundred times and the results were averaged to remove random bias. The main comparison is presented in Figure \ref{fig:results}. The left column corresponds to problem \eqref{eq:problem1} while the right column to problems \eqref{eq:problem2} and \eqref{eq:problem3}. The $x$ axis always depicts the number of input data $n$ chosen in the range $n\in [10^3,10^6]$. The first row depicts the computational time in seconds. The second row depicts the number of evaluations of $h_1,\dots,h_6$.

We observe that the number of evaluations of $h$ in the second row stays relatively constant. Coming back to Table \ref{table:complexity1}, this implies that $\niter$ is constant and the total complexity should be $O(n)$ or $O(n\log n)$. This is confirmed in the first row of Figure \ref{fig:results} where we see the (approximately) linear dependence of time on the data size $n$. To give a more quantitative result, we have interpolated the measured times with function $t(n)=an^b$ for the best possible parameters $a$ and $b$. We show this interpolation in Table \ref{table:complexity2}. We see that the interpolation is close to linear. Note that the larger power of $n$ may hide the logarithm as the domain for $n$ is bounded.

\begin{table}[!ht]
\centering
\caption{Comparison of the observed and theoretical complexity for our methods. In most cases our methods exhibit the complexity of $O(n)$ or $O(n\log n)$ which concurs with Table \ref{table:complexity1}. }
\label{table:complexity2}
\begin{tabular}{@{}lll@{}} \\\toprule
& Observed complexity & Theoretical complexity \\\midrule
\eqref{eq:problem1} with Kullback-Leibler divergence & $3.328\cdot 10^{ -7}n^{1.102}$ & $O(n\niter)$ \\
\eqref{eq:problem2} with $l_1$ norm & $2.794\cdot 10^{ -8}n^{1.125}$ & $O(n\log n)$ \\
\eqref{eq:problem2} with $l_2$ norm & $3.759\cdot 10^{ -7}n^{1.056}$ & $O(n\niter)$ \\
\eqref{eq:problem2} with $l_\infty$ norm & $2.802\cdot 10^{ -7}n^{1.042}$ & $O(n\log n)$ \\ 
\eqref{eq:problem3} & $4.911\cdot 10^{ -7}n^{1.039}$ & $O(n\log n)$ \\  \bottomrule
\end{tabular}
\end{table}

In Figure \ref{fig:comparison} we compare our results to other solvers. We used the package \texttt{JuMP} for Julia \cite{JuMP.jl-2017}. It employed IPOPT for \eqref{eq:problem1} and CPLEX for \eqref{eq:problem2} and \eqref{eq:problem3}. Moreover, for \eqref{eq:problem2} with $l_2$ norm we compared ourselves to the algorithm from \cite{philpott2018distributionally}. To keep the computation possible, we had to reduce the number of points from $n=10^6$ to $n=10^4$. We can see that our algorithms perform significantly better. CPLEX and IPOPT seem to also possess linear complexity unlike the algorithm from \cite{philpott2018distributionally} with quadratic complexity. Its complexity estimate from Table \ref{table:complexity2} was $5.138\cdot 10^{ -7}n^{1.632}$.


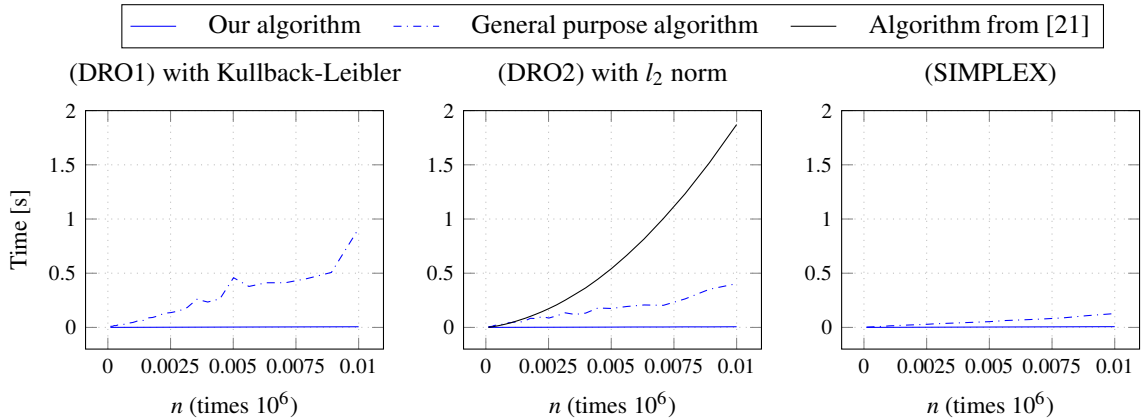
\begin{figure}[!ht]
\begin{tikzpicture}
 \pgfplotsset{small,width=5.5cm,samples=30}
 \begin{groupplot}[group style = {group size = 3 by 1, horizontal sep = 30pt}, grid=major, grid style={dotted, gray!50}, legend cell align={left}, scaled x ticks = false, xtick={0,2500,5000,7500,10000}, xticklabels={0,0.0025,0.005,0.0075,0.01}, scaled y ticks = false, ymin=-0.2, ymax=2]
 \nextgroupplot[title={\eqref{eq:problem1} with Kullback-Leibler}, xlabel={$n$ (times $10^6$)}, ylabel={Time [s]}]
    \addplot [line1] table[x index=1, y index=2] {\tabTime};
    \addplot [line2] table[x index=1, y index=3] {\tabTime};    
 \nextgroupplot[title={\eqref{eq:problem2} with $l_2$ norm}, xlabel={$n$ (times $10^6$)}, ylabel={}, legend style = legendStyleC, legend columns=3]
    \addplot [line1] table[x index=1, y index=16] {\tabTime};
    \addlegendentry{Our algorithm};
    \addplot [line2] table[x index=1, y index=17] {\tabTime};
    \addlegendentry{General purpose algorithm};
    \addplot [line3] table[x index=1, y index=18] {\tabTime};
    \addlegendentry{Algorithm from \cite{philpott2018distributionally}};   
 \nextgroupplot[title={\eqref{eq:problem3} \\}, xlabel={$n$ (times $10^6$)}, ylabel={}]
    \addplot [line1] table[x index=1, y index=19] {\tabTime};
    \addplot [line2] table[x index=1, y index=20] {\tabTime};
 \end{groupplot}
  \node at ($(group c2r1) + (0,2.7)$) {\ref{grouplegendC}};  
\end{tikzpicture}
\caption{Comparison of our method, general-purpose solvers (IPOPT and CPLEX) and the algorithm from \cite{philpott2018distributionally}.} 
\label{fig:comparison}
\end{figure}

\section*{Acknowledgements}

This work was supported by National Natural Science Foundation of China (Grant No. 61850410534), the Program for Guangdong Introducing Innovative and Enterpreneurial Teams (Grant No. 2017ZT07X386), Shenzhen Peacock Plan (Grant No. KQTD2016112514355531) and the Grant Agency of the Czech Republic (18-21409S).

\appendix

\section{Problem \eqref{eq:problem2} with $l_1$ and $l_\infty$ norm}\label{app:norms}

Here we present algorithms for solving \eqref{eq:problem2} for the $l_1$ and $l_\infty$ norms. Since $\bm q$ is a probability distribution due to Assumption \ref{ass}, if we increase some components of $\bm q$, we have to decrease some components of $\bm q$ by the same margin. The priority is on increasing coordinates of $\bm q$ with the lowest value of $\bm c$ while decreasing those with the largest value. We summarize this procedure in Algorithms \ref{algorithm1} and \ref{algorithm2}. Sorting $\bm c$, the lowest values have the lowest indices and similarly for the largest values. Thus, we start with $i=1$ and $j=n$. Then we increase $q_i$ by a possible maximal margin $\delta_1$ and start decreasing $q_j$, $q_{j-1}$ and so on until the total reduction $\delta_2$ equals to $\delta_1$. After doing so, we increase $i$ by one and continue until $i=j$. Note that $\delta_{\rm dec}$ in Algorithm \ref{algorithm1} measures the decrease of $p_j$ while $\delta_{\rm tot}$ in Algorithm \ref{algorithm2} measures the total reduction of $p_j,\dots,p_n$. The first one has to be bounded by $\eps$ while the other one by $\frac{\eps}{2}$.


\begin{figure*}[!ht]
\begin{minipage}{0.47\textwidth}
\begin{algorithm}[H]
    \centering
    \caption{for solving \eqref{eq:problem2} with $p=\infty$}\label{algorithm1}
    \begin{algorithmic}[1]
\Require Sorted array $\bm c$, probabilities $\bm q$, allowed perturbation level $\eps$
\State $\bm p\gets \bm q$, $i\gets 1$, $j\gets n$
\State $\delta_{\rm dec}\gets0$
\While{$i\le j$}
\State $\delta_1\gets\min\{1-p_i,\eps\}$, $\delta_2\gets 0$
\State
\State $p_i\gets p_i+\delta_1$
\While{$\delta_2<\delta_1$}
\If {$\min\{p_j, \eps-\delta_{\rm dec}\}\ge \delta_1-\delta_2$}
\State $p_j\gets p_j - \delta_1 + \delta_2$
\State $\delta_{\rm dec} \gets \delta_{\rm dec} + \delta_1 - \delta_2$
\State \textbf{break} (inner while)
\Else
\State $\delta_2\gets \delta_2+\min\{p_j, \eps-\delta_{\rm dec}\}$
\State $p_j\gets p_j - \min\{p_j, \eps-\delta_{\rm dec}\}$
\State $\delta_{\rm dec}\gets 0$
\State $j\gets j-1$
\If {i==j}
\State $p_i\gets p_i - \delta_1 + \delta_2$
\State \textbf{break} (inner while)
\EndIf
\EndIf
\EndWhile
\EndWhile
\State \textbf{return} $\bm p$
    \end{algorithmic}
\end{algorithm}
\end{minipage}
\hfill
\begin{minipage}{0.47\textwidth}
\begin{algorithm}[H]
    \centering
    \caption{for solving \eqref{eq:problem2} with $p=1$}\label{algorithm2}
    \begin{algorithmic}[1]
\Require Sorted array $\bm c$, probabilities $\bm q$, allowed perturbation level $\eps$
\State $\bm p\gets \bm q$, $i\gets 1$, $j\gets n$
\State $\delta_{\rm tot}\gets 0$
\While{$i\le j$ \textbf{and} $\delta_{\rm tot}\le\frac{\eps}{2}$}
\State $\delta_1\gets\min\{1-p_i,\frac{\eps}{2}-\delta_{\rm tot}\}$, $\delta_2\gets 0$
\State $\delta_{\rm tot}\gets \delta_{\rm tot}+\delta_1$
\State $p_i\gets p_i+\delta_1$
\While{$\delta_2<\delta_1$}
\If {$p_j\ge \delta_1-\delta_2$}
\State $p_j\gets p_j - \delta_1 + \delta_2$
\State
\State \textbf{break} (inner while)
\Else
\State $\delta_2\gets \delta_2+p_j$
\State $p_j\gets 0$
\State
\State $j\gets j-1$
\If {i==j}
\State $p_i\gets p_i - \delta_1 + \delta_2$
\State \textbf{break} (inner while)
\EndIf
\EndIf
\EndWhile
\EndWhile
\State \textbf{return} $\bm p$
    \end{algorithmic}
\end{algorithm}
\end{minipage}
\end{figure*}

\section{Proofs}

In this section, we present all proofs. The proofs are divided into subsections as in the manuscript body. We omit the proof of Theorem \ref{thm7} as it is similar to other proofs.

\subsection{Optimality conditions for Theorems \ref{thm0}-\ref{thm5}}\label{sec:proof1}

We start with a general part which is common to all Theorems \ref{thm0}-\ref{thm5}. Since $\phi$ is convex in $\phi$, since $\eps>0$, since $d(q_i, q_i)=0$ and since $\bm q$ defines a -probability distribution, the Slater constraint qualification is satisfied at $\bm q$. Thus, problem \eqref{eq:problem1} is equivalent to its KKT optimality conditions. The Lagrangian for \eqref{eq:problem1} reads
$$
L(\bm p;\bm  \alpha, \lambda,\mu) = -\sum_{i=1}^nc_ip_i - \sum_{i=1}^n \alpha_ip_i +  \lambda\left(\sum_{i=1}^n p_i-1\right) + \mu\left(\sum_{i=1}^nd(p_i, q_i) - \eps\right).
$$
The minus in front of the first term needs be to present since \eqref{eq:problem1} is a maximization problem. The KKT conditions then amount to the optimality conditions
\begin{subequations}\label{eq:problem1_KKT}
\begin{align}
\frac{\partial L(\cdot)}{\partial p_i} = -c_i- \alpha_i+ \lambda+\mu\nabla_p d(p_i, q_i)  = 0.\label{eq:problem1_KKT_opt1}
\end{align}
the primal feasibility conditions \eqref{eq:problem1}, the dual feasibility conditions $\alpha_i\ge 0$, $ \lambda\in\R$, $\mu\ge 0$ and finally the complementarity conditions
\begin{align}
\alpha_ip_i & =0,\quad\forall i=1,\dots,n,\label{eq:problem1_KKT_comp1} \\
\mu \left(\sum_{i=1}^n d(p_i, q_i)-\eps\right) &= 0.\label{eq:problem1_KKT_comp2}
\end{align}
\end{subequations}

Since $\mu\ge 0$, there are two possibilities.

\paragraph{Case 1 of $\mu=0$:} If $\mu=0$, then from \eqref{eq:problem1_KKT_opt1} we get $\lambda = \alpha_i + c_i$. This, together with $\alpha_i\ge0$ and $\alpha_ip_i=0$ implies that $\alpha_i=0$ for $i\in I$ and that $p_i=0$ for $i\notin I$. Then the KKT system is satisfied if there exists a feasible $\bm p$ with $p_i=0$ for $i\notin I$ such that $\sum_{i=1}^nd(p_i, q_i)\le \eps$. This problem can be verified by solving the convex problem
\begin{equation}\label{eq:check0}
\aligned
\mnmz_{\bm p}\ &\sum_{i\in I}d(p_i,q_i) \\
\st\ &\sum_{i\in I}p_i=1,\\
&0\le p_i,\quad\forall i\in I,
\endaligned
\end{equation}
and checking whether its optimal value is smaller or equal than $\eps - \sum_{i\notin I}d(0,q_i)$.

Since for $\phi$-divergences, we have $d(p_i,q_i)=q_i\phi (\frac{p_i}{q_i})$, it is not difficult to verify that the ratio $\frac{p_i}{q_i}$ is constant. This implies that $\hat{\bm p}$ defined in \eqref{eq:check1} is the solution to \eqref{eq:check0}. This finishes the proof of Theorem \ref{thm0}.

\paragraph{Case 2 of $\mu>0$:} In the opposite case we have $\mu>0$, which due to \eqref{eq:problem1_KKT_comp2} implies that the feasibility conditions change into
\begin{equation}\label{eq:problem1_KKT_feas}
\aligned
\sum_{i=1}^n p_i &= 1, \\
\sum_{i=1}^nd(p_i, q_i) &= \eps.
\endaligned
\end{equation}
We now split the proof into five parts for Theorem \ref{thm1}
-\ref{thm5}.

\begin{proof}[Proof of Theorem \ref{thm1}]
The feasibility constraint may be due to Assumption \ref{ass} written as 
$$
\sum_{i=1}^n\left(p_i\log\left(\frac{p_i}{q_i}\right) - p_i \right) = \eps-1.
$$
For now we assume that $p_i>0$ for all $i$ and remove this assumption later. This implies $\alpha_i=0$. Then the optimality condition \eqref{eq:problem1_KKT_opt1} reads
$$
-c_i + \lambda + \mu\log p_i - \mu\log q_i = 0,
$$
from which we deduce
\begin{equation}\label{eq:proof1_1}
p_i = q_i \exp\left(\frac{c_i-\lambda}{\mu}\right).
\end{equation}
Plugging \eqref{eq:proof1_1} into the feasibility conditions \eqref{eq:problem1_KKT_feas} yields
\begin{align}
\label{eq:proof1_2} \sum_{i=1}^n q_i \exp\left(\frac{c_i-\lambda}{\mu}\right) &= 1, \\
\label{eq:proof1_3} \sum_{i=1}^n q_i \exp\left(\frac{c_i-\lambda}{\mu}\right)\frac{c_i-\lambda}{\mu} &= \eps.
\end{align}
We can express $\mu$ from \eqref{eq:proof1_2} via
$$
\sum_{i=1}^nq_i\exp\left(\frac{c_i}{\mu}\right) = \exp\left(\frac{\lambda}{\mu}\right),
$$
which together with \eqref{eq:proof1_3} gives the final equation \eqref{eq:thm1_1}. The optimal probabilities \eqref{eq:thm1_2} then follow from \eqref{eq:proof1_1}.


Now we need to remove the assumption of $p_i>0$. Recall that
$$
h_1(\mu) = \sum_{i=1}^nq_i \exp\left(\frac{c_i}{\mu}\right)\left(\frac{c_i}{\mu} - \log\left(\sum_{j=1}^nq_j\exp\left(\frac{c_j}{\mu}\right)\right) - \eps\right). $$
For its middle part we have
$$
\frac{c_i}{\mu} - \log\left(\sum_{j=1}^nq_j\exp\left(\frac{c_j}{\mu}\right)\right) \le \frac{c_i}{\mu} - \log\left(\sum_{j=1}^nq_j\exp\left(\frac{\cmin}{\mu}\right)\right) = \frac{c_i}{\mu} - \frac{\cmin}{\mu} \le \frac{\cmax-\cmin}{\mu},
$$
which implies that
\begin{equation}\label{eq:proof1_4}
h_1(\mu) \le \sum_{i=1}^nq_i \exp\left(\frac{c_i}{\mu}\right)\left(\frac{\cmax-\cmin}{\mu} - \eps\right) \le 0\qquad\text{whenever}\qquad \mu\ge\frac{\cmax-\cmin}{\eps}.
\end{equation}

We consider now the limit of $h_1(\mu)$ as $\mu\downarrow 0$. Due to the properties of the exponential function, for all $\alpha>0$, there is some $\mu_0$ such that for all $\mu\in(0,\mu_0)$ we have
$$
- \log\left(\sum_{j=1}^nq_j\exp\left(\frac{c_j}{\mu}\right)\right) \ge - \log\left((1+\alpha)\sum_{j\in I}q_j\exp\left(\frac{c_j}{\mu}\right)\right) = -\log(1+\alpha) - \log\left(\sum_{j\in I}q_j\right) - \frac{\cmax}{\mu}.
$$
This implies
$$
\aligned
h_1(\mu) &\ge \sum_{i=1}^nq_i \exp\left(\frac{c_i}{\mu}\right)\left(\frac{c_i}{\mu} - \log(1+\alpha) - \log\left(\sum_{j\in I}q_j\right) - \frac{\cmax}{\mu} - \eps\right).
\endaligned
$$
The right-most term is positive and independent of $\mu$ whenever $i\in I$ and $\alpha$ is sufficiently small due to the assumptions of Theorem \ref{thm1}. Moreover as
$$
\exp\left(\frac{\cmax}{\mu}\right) \gg \exp\left(\frac{c_i}{\mu}\right)\frac{1}{\mu}
$$
for all $i\notin I$, we deduce that $h_1(\mu)\to\infty$ as $\mu\downarrow 0$. This combined with \eqref{eq:proof1_4} and the continuity of $h_1$ implies that the equation $h_1(\mu)=0$ has a solution on $(0,\frac{\cmax-\cmin}{\eps}]$. Since the optimality conditions are equivalent to problem \eqref{eq:problem1} due to convexity, the existence of solution also implies that the assumption of $p_i>0$ may be alleviated.
\end{proof}

\begin{proof}[Proof of Theorem \ref{thm2}]
The form of $\phi_2$ implies $p_i>0$ for all $i$, which further means $\alpha_i=0$. Then the optimality condition \eqref{eq:problem1_KKT_opt1} reads
$$
-c_i + \lambda - \mu\frac{q_i}{p_i} = 0,
$$
from which we deduce
\begin{equation}\label{eq:proof2_1}
p_i = q_i \frac{\mu}{\lambda-c_i}.
\end{equation}
Plugging \eqref{eq:proof2_1} into the feasibility conditions \eqref{eq:problem1_KKT_feas} yields
\begin{align}
\label{eq:proof2_2} \sum_{i=1}^nq_i\frac{\mu}{\lambda-c_i} &= 1, \\
\label{eq:proof2_3} \sum_{i=1}^n q_i\log\left(\lambda-c_i\right) &= \eps + \log\mu.
\end{align}
We can express $\mu$ from \eqref{eq:proof2_2} via
$$
\sum_{i=1}^n\frac{q_i}{\lambda-c_i} = \frac{1}{\mu},
$$
which together with \eqref{eq:proof2_3} gives the final equation \eqref{eq:thm2_1}. The optimal probabilities \eqref{eq:thm2_2} then follow from \eqref{eq:proof2_1}. The constraint $\mu>0$ transfers to $\lambda>\cmax$ due to \eqref{eq:proof2_1}.


Now we are interested in the limits. Recall that
$$
h_2(\lambda) = \sum_{i=1}^n q_i\log\left(\lambda-c_i\right) + \log\left(\sum_{i=1}^n\frac{q_i}{\lambda-c_i}\right) - \eps.
$$
Then we have
$$
\aligned
h_2(\lambda )&\le \sum_{i=1}^n q_i\log\left(\lambda-\cmin\right) + \log\left(\sum_{i=1}^n\frac{q_i}{\lambda-\cmax}\right) - \eps = \log\left(\lambda-\cmin\right) - \log\left(\lambda-\cmax\right) - \eps \\
&= \log\left(\frac{\lambda-\cmin}{\lambda-\cmax}\right) - \eps = \log\left(1+\frac{\cmax-\cmin}{\lambda-\cmax}\right) - \eps \le \frac{\cmax-\cmin}{\lambda-\cmax} - \eps  
\endaligned
$$
Thus 
\begin{equation}\label{eq:proof2_4}
h_2(\lambda)\le 0 \qquad\text{whenever}\qquad \lambda\ge \cmax+\frac{\cmax-\cmin}{\eps}.
\end{equation}

We consider now the limit of $h_2(\lambda)$ as $\lambda\downarrow \cmax$. Denoting $ \csmax$ the second largest distinct component value of $\bm c$, we have
$$
\aligned
h_2(\lambda) &\ge \sum_{i\in I} q_i\log\left(\lambda-c_i\right) + \sum_{i\notin I} q_i\log\left(\lambda-c_i\right) + \log\left(\sum_{i\in I}\frac{q_i}{\lambda-c_i}\right) - \eps \\
&\ge \sum_{i\in I} q_i\log\left(\lambda-c_i\right) + \sum_{i\notin I} q_i\log\left(\cmax- \csmax\right) + \log\left(\sum_{i\in I}\frac{q_i}{\lambda-c_i}\right) - \eps \\
&= \left(\sum_{i\in I}q_i-1\right) \log\left(\lambda-\cmax\right) + \sum_{i\notin I} q_i\log\left(\cmax- \csmax\right) + \log\left(\sum_{i\in I} q_i\right) - \eps \\
\endaligned
$$
Due to the assumptions of Theorem \ref{thm5}, we obtain $h_2(\lambda)\to\infty$ as $\lambda\downarrow\cmax$. This combined with \eqref{eq:proof2_4} and the continuity of $h_2$ implies that the equation $h_2(\lambda)=0$ has a solution on $(\cmax,\cmax+\frac{\cmax-\cmin}{\eps}]$.
\end{proof}

\begin{proof}[Proof of Theorem \ref{thm3}]
The feasibility constraint may due to Assumption \ref{ass} be written as 
$$
-2\sum_{i=1}^n\sqrt{p_iq_i} = \eps - 2.
$$
For now we assume that $p_i>0$ for all $i$ and remove this assumption later. This implies $\alpha_i=0$. Then the optimality condition \eqref{eq:problem1_KKT_opt1} reads
\begin{equation}\label{eq:proof3_0}
-c_i + \lambda - \mu \sqrt{\frac{q_i}{p_i}} = 0.
\end{equation}
from which we deduce
\begin{equation}\label{eq:proof3_1}
p_i = q_i \frac{\mu^2}{(\lambda-c_i)^2}.
\end{equation}
Plugging \eqref{eq:proof3_1} into the feasibility conditions \eqref{eq:problem1_KKT_feas} yields
\begin{align}
\label{eq:proof3_2} \sum_{i=1}^n q_i \frac{\mu^2}{(\lambda-c_i)^2} &= 1, \\
\label{eq:proof3_3} 2\sum_{i=1}^nq_i \frac{\mu}{\nrm{\lambda-c_i}} &= 2-\eps.
\end{align}
We can express $\mu$ from \eqref{eq:proof3_2} via
$$
\sum_{i=1}^n \frac{q_i}{(\lambda-c_i)^2} = \frac{1}{\mu^2},
$$
which together with \eqref{eq:proof3_3} gives the final equation \eqref{eq:thm3_1}. The optimal probabilities \eqref{eq:thm3_2} then follow from \eqref{eq:proof3_1}. The constraint $\mu>0$ transfers to $\lambda>\cmax$ due to \eqref{eq:proof3_0}, which also allows us to remove the absolute value from \eqref{eq:proof3_3}.


Now we need to remove the assumption of $p_i>0$. Recall that
$$
h_3(\lambda) = 2\sum_{i=1}^n \frac{q_i}{\lambda-c_i} - (2-\eps)\sqrt{\sum_{i=1}^n \frac{q_i}{(\lambda-c_i)^2}}
$$
Then we have
$$
\aligned
h_3(\lambda) &\ge 2\sum_{i=1}^n \frac{q_i}{\lambda-\cmin} - (2-\eps)\sqrt{\sum_{i=1}^n \frac{q_i}{(\lambda-\cmax)^2}} = \frac{2}{\lambda-\cmin} - \frac{2-\eps}{\lambda-\cmax} \\
&= \frac{2(\lambda-\cmax) - (2-\eps)(\lambda-\cmin)}{(\lambda-\cmin) (\lambda-\cmax)} = \frac{\eps\lambda - 2(\cmax-\cmin) - \eps\cmin}{(\lambda-\cmin) (\lambda-\cmax)}
\endaligned
$$
Thus 
\begin{equation}\label{eq:proof3_4}
h_3(\lambda)\ge 0 \qquad\text{whenever}\qquad \lambda\ge \cmin+\frac{2(\cmax-\cmin)}{\eps} = \cmax + \frac{(2-\eps)(\cmax-\cmin)}{\eps}.
\end{equation}

We consider now the limit of $h_3(\lambda)$ as $\lambda\downarrow \cmax$. Denoting $ \csmax$ the second largest distinct component value of $\bm c$, we have
$$
\aligned
h_3(\lambda) &\le 2\sum_{i\in I}\frac{q_i}{\lambda-c_i} + 2\sum_{i\notin I}\frac{q_i}{\lambda-c_i}  - (2-\eps)\sqrt{\sum_{i\in I} \frac{q_i}{(\lambda-c_i)^2}} \\
&\le 2\frac{\sum_{i\in I} q_i}{\lambda-\cmax} + \frac{2}{\cmax -  \csmax}  - (2-\eps) \frac{\sqrt{\sum_{i\in I} q_i}}{\lambda-\cmax} \\
&\le \frac{2\sum_{i\in I} q_i - (2-\eps)\sqrt{\sum_{i\in I} q_i}}{\lambda-\cmax} + \frac{2}{\cmax -  \csmax}\\
\endaligned
$$
Due to the assumptions of Theorem \ref{thm3}, we obtain $h_3(\lambda)\to-\infty$ as $\lambda\downarrow\cmax$. This combined with \eqref{eq:proof3_4} and the continuity of $h_3$ implies that the equation $h_3(\lambda)=0$ has a solution on $(\cmax,\cmax+\frac{(2-\eps)(\cmax-\cmin)}{\eps}]$. Since the optimality conditions are equivalent to problem \eqref{eq:problem1} due to convexity, the existence of solution also implies that the assumption of $p_i>0$ may be alleviated.
\end{proof}

\begin{proof}[Proof of Theorem \ref{thm4}]
The form of $\phi_4$ implies $p_i>0$ for all $i$, which further means $\alpha_i=0$. Due to Assumption \ref{ass}, the feasibility constraints on distance amounts to
$$
\sum_{i=1}^n\frac{q_i^2}{p_i} = 1+\eps.
$$
Then the optimality condition \eqref{eq:problem1_KKT_opt1} reads
$$
-c_i + \lambda - \mu\frac{q_i^2}{p_i^2} = 0.
$$
from which we deduce
\begin{equation}\label{eq:proof4_1}
p_i = q_i\sqrt{\frac{\mu}{\lambda-c_i}}.
\end{equation}
Plugging \eqref{eq:proof4_1} into the feasibility conditions \eqref{eq:problem1_KKT_feas} yields
\begin{align}
\label{eq:proof4_2} \sum_{i=1}^nq_i\sqrt{\frac{\mu}{\lambda-c_i}} &= 1, \\
\label{eq:proof4_3} \sum_{i=1}^nq_i\sqrt{\frac{\lambda-c_i}{\mu}} &= 1+\eps.
\end{align}
We can express $\mu$ from \eqref{eq:proof4_3} via
$$
\frac{1}{1+\eps}\sum_{i=1}^nq_i\sqrt{\lambda-c_i} = \sqrt{\mu},
$$
which together with \eqref{eq:proof4_2} gives the final equation \eqref{eq:thm4_1}. The optimal probabilities \eqref{eq:thm4_2} then follow from \eqref{eq:proof4_1}. The constraint $\mu>0$ transfers to $\lambda>\cmax$ due to \eqref{eq:proof4_1}.
\end{proof}

\begin{proof}[Proof of Theorem \ref{thm5}]
The feasibility constraint may due to Assumption \ref{ass} be written as 
$$
\sum_{i=1}^n\frac{p_i^2}{q_i} = 1+\eps.
$$
Then the optimality condition \eqref{eq:problem1_KKT_opt1} reads (for the uniformity of results, we flipped the sign of $\lambda$)
$$
-c_i - \alpha_i - \lambda + 2\mu\frac{p_i}{q_i} = 0.
$$
from which we due to the complementarity conditions \eqref{eq:problem1_KKT_comp1} deduce
\begin{equation}\label{eq:proof5_1}
p_i = \frac{1}{2\mu}q_i \max\left(c_i+\lambda,0\right).
\end{equation}
Plugging \eqref{eq:proof5_1} into the feasibility conditions \eqref{eq:problem1_KKT_feas} yields
\begin{align}
\label{eq:proof5_2} \frac{1}{2\mu}\sum_{i=1}^nq_i \max\left(c_i+\lambda,0\right) &= 1, \\
\label{eq:proof5_3} \frac{1}{4\mu^2}\sum_{i=1}^nq_i {\max}^2\left(c_i+\lambda,0\right) &= 1+\eps.
\end{align}
We can express $\mu$ from \eqref{eq:proof5_2} via
$$
\frac{1}{2}\sum_{i=1}^nq_i \max\left(c_i+\lambda,0\right) = \mu,
$$
which together with \eqref{eq:proof5_3} gives the final equation \eqref{eq:thm5_1}. The optimal probabilities \eqref{eq:thm5_2} then follow from \eqref{eq:proof5_1}. The constraint $\mu>0$ transfers to $\lambda>-\cmax$ due to \eqref{eq:proof5_1}.
\end{proof}

\subsection{Optimality conditions for Theorem \ref{thm6}}

To prove Theorem \ref{thm6}, we first realize that the first part of Section \ref{sec:proof1} including the first paragraph of ``Case 1 of $\mu=0$'' holds true with $d(p_i,q_i)=\frac12(p_i-q_i)^2$. Moreover, problem \eqref{eq:check0} takes the form
\begin{equation}\label{eq:check00}
\aligned
\mnmz_{\bm p}\ &\frac12\sum_{i\in I}(p_i-q_i)^2 \\
\st\ &\sum_{i\in I}p_i=1,\\
&0\le p_i,\quad\forall i\in I.
\endaligned
\end{equation}
Since $q_i>0$ and $\sum_{i\in I}q_i<1$ due to Assumption \ref{ass}, it is not difficult to verify that $\hat{\bm p}$ defined in \eqref{eq:check3} is the optimal solution of \eqref{eq:check00}. This proves the first part of Theorem \ref{thm6}.

For the second part, we realize that the feasibility constraint may be written as 
$$
\frac12\sum_{i=1}^n\left(p_i - q_i\right)^2 = \frac12\eps^2.
$$
Then the optimality condition \eqref{eq:problem1_KKT_opt1} reads
$$
-c_i- \alpha_i+ \lambda+\mu(p_i-q_i) = 0,
$$
from which we due to the primal feasibility condition $p_i\ge0$, the dual feasibility conditions $\alpha_i\ge 0$ and the complementarity condition \eqref{eq:problem1_KKT_comp1} deduce
\begin{equation}\label{eq:proof6_1}
p_i = \max\left(q_i - \frac{1}{\mu}(\lambda-c_i), 0\right).
\end{equation}
Plugging \eqref{eq:proof6_1} into the feasibility conditions \eqref{eq:problem1_KKT_feas} yields
\begin{align}
\label{eq:proof6_2} \sum_{i=1}^n\max\left(q_i - \frac{1}{\mu}(\lambda-c_i), 0\right) &= 1, \\
\label{eq:proof6_3} \sum_{i=1}^n\left( \max\left(q_i - \frac{1}{\mu}(\lambda-c_i), 0\right) - q_i \right)^2&= \eps^2.
\end{align}
The final equations \eqref{eq:thm6_1} are obtained from \eqref{eq:proof6_2} and \eqref{eq:proof6_3} by simple calculus, the formula $\max(-x,-y)=-\min(x,y)$ and Assumption \ref{ass}.

\subsection{Convexity for Proposition \ref{prop:convex}}

We first show an auxiliary result which states the continuity on $h_6$.

\begin{lemma}\label{lemma:g6}
For each $\mu>0$ there is a unique $\lambda(\mu)$ which solves $g_6(\lambda;\mu)$. Moreover, function $h_6(\mu)$ is continuous on $(0,\infty)$.
\end{lemma}
\begin{proof}
Fix any $\mu>0$ and recall that
$$
g_6 (\lambda;\mu) = \sum_{i=1}^n\min\left(\lambda-c_i, \mu q_i\right).
$$
Since $\mu>0$ for at least one $i$ we have $\lambda(\mu)-c_i<\mu q_i$ which implies that $g_6$ is strictly increasing in $\lambda$ around $\lambda(\mu)$. Thus, the solution $\lambda(\mu)$ is unique.

Consider now any $\mu_k\to \mu>0$. Since the corresponding $\lambda(\mu_k)$ are bounded, we may select a converging subsequence, say $\lambda(\mu_k)\to\lambda^*$. Then
$$
0 = \sum_{i=1}^n\min(\lambda(\mu_k)-c_i,\mu_kq_i) \to \sum_{i=1}^n\min(\lambda^*-c_i,\mu q_i),
$$
which implies that $\lambda^*=\lambda(\mu)$. Thus, $\lambda$ is a continuous function of $\mu$. This further implies that $h_6$ is a continuous function of $\mu$. 
\end{proof}

Concerning the proof of Proposition \ref{prop:convex}, we divide it into three parts.

\paragraph{Proof for $h_4$:}

After rearranging of terms, we get
\begin{equation}\label{eq:h4_rearrange}
h_4(\lambda) = \sum_{i=1}^n\sum_{j=1}^nq_iq_j g_{ij}(\lambda) - 1 - \eps,
\end{equation}
where for $i,j=1,\dots,n$ we define
$$
g_{ij}(\lambda) := (\lambda-c_i)^{\frac12} (\lambda-c_j)^{-\frac12}.
$$
Computing the first derivative
$$
\aligned
g_{ij}'(\lambda) &= \frac12(\lambda-c_i)^{-\frac12} (\lambda-c_j)^{-\frac12} - \frac12(\lambda-c_i)^{\frac12} (\lambda-c_j)^{-\frac32} \\
&= \frac12(\lambda-c_i)^{-\frac12} (\lambda-c_j)^{-\frac32}(\lambda - c_j - (\lambda - c_i)) \\
&= \frac12(\lambda-c_i)^{-\frac12} (\lambda-c_j)^{-\frac32}(c_i - c_j)
\endaligned
$$
and the second derivative
$$
\aligned
g_{ij}''(\lambda) &= \frac12(c_i - c_j)\left(-\frac12(\lambda-c_i)^{-\frac32} (\lambda-c_j)^{-\frac32} - \frac32(\lambda-c_i)^{-\frac12} (\lambda-c_j)^{-\frac52} \right) \\
&= -\frac14(c_i - c_j)\left((\lambda-c_i)^{-\frac32} (\lambda-c_j)^{-\frac32} + 3(\lambda-c_i)^{-\frac12} (\lambda-c_j)^{-\frac52} \right) \\
&= -\frac14(c_i - c_j)(\lambda-c_i)^{-\frac32}(\lambda-c_j)^{-\frac52}( \lambda-c_j + 3(\lambda-c_i) ) \\
&= -\frac14(c_i - c_j)(\lambda-c_i)^{-\frac32}(\lambda-c_j)^{-\frac52}( 4\lambda -3c_i - c_j).
\endaligned
$$
we realize that
$$
\aligned
g_{ij}''(\lambda) + g_{ji}''(\lambda) &= -\frac14(c_i - c_j)(\lambda-c_i)^{-\frac52}(\lambda-c_j)^{-\frac52}\big( 4\lambda -3c_i - c_j)(\lambda-c_i) - (4\lambda -c_i - 3c_j)(\lambda-c_j)\big) \\
&= -\frac14(c_i - c_j)(\lambda-c_i)^{-\frac52}(\lambda-c_j)^{-\frac52}\big( 3(c_j-c_i)(2\lambda-c_i-c_j)\big) \\
&= \frac34(c_i - c_j)^2(\lambda-c_i)^{-\frac52}(\lambda-c_j)^{-\frac52} (2\lambda-c_i-c_j) \ge 0.
\endaligned
$$
This together with \eqref{eq:h4_rearrange} implies that $h_4$ is convex.

\paragraph{Proof for $h_5$:}

Recall that
\begin{equation}\label{eq:h5_1}
h_5(\lambda) = \sum_{i=1}^nq_i {\max}^2\left(c_i+\lambda,0\right) - (1+\eps)\left(\sum_{i=1}^nq_i \max\left(c_i+\lambda,0\right)\right)^2
\end{equation}
For simplicity assume that $-c_1<\dots<-c_n$. Then on $\lambda\in(-c_j,-c_{j+1})$ we have
$$
h_5(\lambda) = \sum_{i=1}^jq_i (c_i+\lambda)^2 - (1+\eps)\left(\sum_{i=1}^jq_i (c_i+\lambda)\right)^2
$$
The derivative at this interval equals to
\begin{equation}\label{eq:h5_2}
\aligned
h_5'(\lambda) &= 2\sum_{i=1}^jq_i (c_i+\lambda) - 2(1+\eps)\left(\sum_{i=1}^jq_i (c_i+\lambda)\right)\sum_{i=1}^jq_i \\
&= 2\sum_{i=1}^jq_i (c_i+\lambda)\left(1 - (1+\eps)\sum_{i=1}^jq_i\right) \\
&= 2\left(\sum_{i=1}^nq_i \max\left(c_i+\lambda,0\right)\right) \left(1 - (1+\eps)\sum_{i=1}^jq_i\right).
\endaligned
\end{equation}

Due to the assumption of $-c_1<\dots<-c_n$, the violation of \eqref{eq:check2} implies $(1+\eps)q_1<1$. This due to \eqref{eq:h5_2} means that $h_4'(\lambda) > 0$ on $\lambda\in(-c_1,-c_2)$. Since $h_5(-c_1)=0$ due to \eqref{eq:h5_1}, this implies that $h_5$ is positive on $(-c_1,-c_2)$. Moreover, $h_5'$ equals to a product of two terms, the first one is always positive while the second one is piecewise constant with decreasing values on individual pieces. This implies that there is some $\lambda_0$ such that $h_5$ is non-decreasing on $(-\cmax,\lambda_0)$ and decreasing on $(\lambda_0,\infty)$. By the same arguments we have that $h_5'$ is decreasing on $(\lambda_0,\infty)$, which implies that $h_5$ is concave on this interval. This finishes the proof of the second part.

If $-c_1\le\dots\le-c_n$ instead of the assumed $-c_1<\dots<-c_n$, then the proof can be performed in exactly the same way but $(1+\eps)q_1<1$ changes to $(1+\eps)\sum_{i=1}^jq_i<1$, where $j$ is the cardinaly of $I$. Then we can get the same estimates of \eqref{eq:h5_2} as in the previous paragraphs.

\paragraph{Proof for $h_6$:}

Define
\begin{equation}\label{eq:thm6_I}
I(\mu) = \{i\mid \lambda(\mu)-c_i < \mu q_i\}
\end{equation}
and consider any $0<\mu_1<\mu_2$. Since $\lambda(\mu_1)$ solves \eqref{eq:thm6_12} for $\mu=\mu_1$ and $\lambda(\mu_2)$ solves the same equation for $\mu=\mu_2$, we obtain $\lambda(\mu_1)\ge \lambda(\mu_2)$. This due to \eqref{eq:thm6_I} implies $I(\mu_1)\subset I(\mu_2)$. Thus, $I(\mu)$ is a non-decreasing function (with respect to set inclusion) of $\mu$.

From \eqref{eq:thm6_12} we have
$$
\aligned
0 &= \sum_{i=1}^n\min\left(\lambda(\mu)-c_i, \mu q_i\right)  = \sum_{i\in I(\mu)}(\lambda(\mu)-c_i) + \sum_{i\notin I(\mu)} \mu q_i \\
&= \nrm{I(\mu)}\lambda(\mu) -\sum_{i\in I(\mu)}c_i + \mu \sum_{i\notin I(\mu)} q_i,
\endaligned
$$
from which we deduce
\begin{equation}\label{eq:problem1_proof1}
\lambda(\mu) = \frac{1}{\nrm{I(\mu)}}\left(\sum_{i\in I(\mu)} c_i - \mu\sum_{i\notin I(\mu)}q_i\right).
\end{equation}
Since $I(\mu)$ is a non-decreasing function of $\mu$, this implies that $\lambda(\mu)$ is a piecewise linear function with a finite number of pieces. On each of these pieces, we have
\begin{equation}\label{eq:problem1_proof2}
\lambda'(\mu) = -\frac{1}{\nrm{I(\mu)}}\sum_{i\notin I(\mu)}q_i
\end{equation}
and consequently
$$
h_6(\mu) = \sum_{i\in I(\mu)}(\lambda(\mu)-c_i)^2 + \sum_{i\notin I(\mu)} \mu^2q_i^2 - \eps^2\mu^2.
$$
Differentiating this relation yields
\begin{equation}\label{eq:thm6_der}
\aligned
\frac12 h_6'(\mu) &= \lambda'(\mu)\sum_{i\in I(\mu)}(\lambda(\mu)-c_i) + \mu\sum_{i\notin I(\mu)} q_i^2 - \mu\eps^2 \\
&= -\lambda'(\mu) \sum_{i\in I(\mu)}c_i + \lambda'(\mu) \nrm{I(\mu)} \lambda(\mu) + \mu\sum_{i\notin I(\mu)} q_i^2 - \mu \eps^2 \\
&= -\lambda'(\mu) \sum_{i\in I(\mu)}c_i + \lambda'(\mu) \left(\sum_{i\in I(\mu)} c_i - \mu\sum_{i\notin I(\mu)}q_i\right) + \mu\sum_{i\notin I(\mu)} q_i^2 - \mu \eps^2 \\
&= - \mu\lambda'(\mu) \sum_{i\notin I(\mu)}q_i + \mu\sum_{i\notin I(\mu)} q_i^2 - \mu \eps^2 \\
&= \mu\left(\frac{1}{\nrm{I(\mu)}}\sum_{i\notin I(\mu)}q_i\sum_{i\notin I(\mu)}q_i + \sum_{i\notin I(\mu)} q_i^2 - \eps^2\right),
\endaligned
\end{equation}
where in the third equality we used \eqref{eq:problem1_proof1} and in the last one \eqref{eq:problem1_proof2}. 

Denote $ \csmax$ the second largest distinct component value of $\bm c$ and define $\hat\lambda = \frac12(\cmax +  \csmax)$. Fix any $\mu$ sufficiently small but positive. Then we have
$$
\aligned
g_6(\hat\lambda,\mu) &= \sum_{i=1}^n\min\left(\hat \lambda -c_i, \mu q_i\right) = \sum_{i\in I}\min\left(\hat \lambda -c_i, \mu q_i\right) + \sum_{i\notin I}\min\left(\hat \lambda -c_i, \mu q_i\right) \\
&= \frac12\nrm{I}( \csmax - \cmax) + \sum_{i\notin I}\min\left(\hat \lambda -c_i, \mu q_i\right) \le \frac12\nrm{I}( \csmax - \cmax) + \mu\sum_{i\notin I}q_i < 0
\endaligned
$$
whenever $\mu$ is sufficiently small. Since $g_6(\lambda(\mu);\mu)=0$ due to definition and since $g_6$ is non-decreasing in $\lambda$, this means that $\lambda(\mu)\ge \hat\lambda = \frac12(\cmax +  \csmax)$. This due to \eqref{eq:thm6_I} implies $I(\mu)=I$ whenever $\mu$ is sufficiently small. It is possible to show that the violation of \eqref{eq:check4} is equivalent to
$$
\frac{1}{\nrm{I}}\sum_{i\notin I}q_i\sum_{i\notin I}q_i + \sum_{i\notin I} q_i^2 - \eps^2 > 0.
$$
Combining this with \eqref{eq:thm6_der} and $I(\mu)=I$, this implies that $h_6$ is strictly increasing on some interval $(0,\mu_1)$. Moreover, since $\lambda(0)=\cmax$, we have $I(0)=\emptyset$ and thus $h_6(0)=0$. This implies that $h_6$ is positive on $(0,\mu_1)$. 

Moreover, $h_6'$ equals to a product of two terms, the first one is always positive while the second one is piecewise constant with decreasing values on individual pieces. This implies that there is some $\mu_0$ such that $h_6$ is non-decreasing on $(0,\mu_0)$ and decreasing on $(\mu_0,\infty)$. By the same arguments, we have that $h_6'$ is decreasing on $(\mu_0,\infty)$, which implies that $h_6$ is concave on this interval. This finishes the proof of the second part.

\section{Convergence of Newton's method}

We show the following theorem only for differentiable functions. However, it holds for any concave function by replacing the derivative by its concave superdifferential.

\begin{lemma}\label{lemma:newton}
Consider a continuous concave function $h:[a,b]\to\R$ with $h(a)>0$ and $h(b)<0$. Then the Newton's method
$$
\lambda^{k+1} = \lambda^k - \frac{h(\lambda^k)}{h'(\lambda^k)}
$$
started from any point $\lambda^0$ with $h(\lambda^0)<0$ gives a decreasing sequence which converges to some $\bar \lambda\in(a,b)$ with $h(\bar \lambda)=0$.
\end{lemma}
\begin{proof}
Due to concavity of $h$ and $h(a)>0$ with $h(b)<0$, there exists unique $\lambda^*\in(a,b)$ with $h(\lambda^*)=0$. Moreover, due to concavity again we have $h'(\lambda^*)<0$. Since $h$ is concave and since $h(a)>0$ and $h(\lambda^0)<0$, we obtain $\lambda^0>\lambda^*$. Then the Newton's method forms a decreasing sequence $\{\lambda^k\}$ bounded below by $\lambda^*$, which is therefore convergent. At the same time, $h'(\lambda^k)$ is uniformly bounded above by zero as $h'(\lambda^k)<h'(\lambda^*)<0$. This implies
$$
\frac{h(\lambda^k)}{h'(\lambda^k)} = \lambda^{k} - \lambda^{k+1} \to 0 ,
$$
which due to the uniform boundedness of $h'(\lambda^k)$ from zero implies $h(\lambda^k)\to 0$. But this due to the continuity of $h$ means that $\lambda^k\to\bar\lambda$.
\end{proof}

\bibliographystyle{abbrv}
\bibliography{References}

\end{document}